[1994/12/01]
\documentclass{ijmart-mod}
\chardef\bslash=`\\ 

\usepackage{graphicx}
\usepackage[breaklinks=true]{hyperref}
\usepackage{mathtools}

\newtheorem{thm}{Theorem}[section]
\newtheorem{cor}[thm]{Corollary}
\newtheorem{lem}[thm]{Lemma}
\newtheorem{prop}[thm]{Proposition}

\theoremstyle{definition}

\newtheorem{rem}[thm]{Remark}

\theoremstyle{remark}

\newcommand{\eval}[2][\right]{\relax
  \ifx#1\right\relax \left.\fi#2#1\rvert}

\begin{document}
\title{Polytopal balls arising in optimization}

\author[A. Deza]{Antoine Deza}
\address{McMaster University, Hamilton, Ontario, Canada}
\email{deza@mcmaster.ca} 

\author[J.-B. Hiriart-Urruty]{Jean-Baptiste Hiriart-Urruty}
\address{Universit{\'e} Paul Sabatier, Toulouse, France}
\email{jean-baptiste.hiriart-urruty@math.univ-toulouse.fr}

\author[L. Pournin]{Lionel Pournin}
\address{Universit{\'e} Paris 13, Villetaneuse, France}
\email{lionel.pournin@univ-paris13.fr}

\begin{abstract}
We study a family of polytopes and their duals, that appear in various optimization problems as the unit balls for certain norms. These two families interpolate between the hypercube, the unit ball for the $\infty$-norm, and its dual cross-polytope, the unit ball for the $1$-norm. We give combinatorial and geometric properties of both families of polytopes such as their $f$-vector, their volume, and the volume of their boundary.
\end{abstract}
\maketitle

\section{Introduction}\label{HCP.sec.0}

A family of norms on $\mathbb{R}^d$ resembling the usual Euclidean norm, yet polytopal in the sense that the balls for these norms are polytopes, were introduced in \cite{Watson1992} as a tool to solve linear approximation problems. These norms, defined by
\begin{equation}\label{HCP.sec.0.eq.-1}
\|x\|_{(k)}=\inf\{\|u\|_1+k\|v\|_\infty:u+v=x\}\mbox{,}
\end{equation}
where $x$ is a vector from $\mathbb{R}^d$ and $k$ is a parameter that belongs to the interval $[1,d]$, were later considered in the context of robust optimization \cite{BertsimasPachamanovaSim2004}, a method to deal with linear optimization under uncertain constraints. As shown in \cite{BertsimasPachamanovaSim2004}, defining the uncertainty constraints using these norms, referred to as $D$-norms in this case, allows for an efficient way to solve robust optimization problems. It is further observed in \cite{GotohUryasev2016,PavlikovUryasev2014} that these norms are naturally connected with the \emph{conditional value at risk}, a popular metric used in quantitative finance: just as the conditional value at risk, these norms, called the CVaR norms in this other context, put the emphasis on the largest coordinates of a vector from $\mathbb{R}^d$. In particular, it is shown in \cite{PavlikovUryasev2014} that these norms are a solution to an optimization problem regarding the conditional value at risk. The same norms also appear in optimization problems over sets of matrices \cite{WuDingSunToh2014}, where they are called vector $k$-norms, and in sparse optimization \cite{GaudiosoGorgoneHiriart-Urruty2020,GotohTakedaTono2018}. In the latter case, one is faced with the \emph{sparsity} constraint on the solutions to a problem: the desired solutions---vectors from $\mathbb{R}^d$---are required to have a prescribed number of non-zero coordinates. This happens for instance in data science, in machine learning \cite{PokuttaSpiegelZimmer2020}, in mathematical imaging, or in statistics among other fields. The number of non-zero coordinates of a vector $x$ of $\mathbb{R}^d$ is often denoted by $\|x\|_0$ in the optimization literature. Formally,
$$
\|x\|_0=\left|\{i:x_i\neq0\}\right|\mbox{,}
$$
where $x_1$ to $x_d$ denote the coordinates of $x$.
\begin{figure}
\begin{center}
\includegraphics{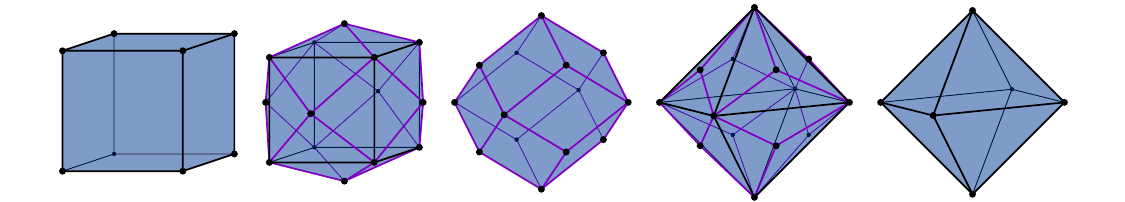}
\caption{The polytope $\rho_{3,k}$, when $k$ is equal to $1$, $3/2$, $2$, $5/2$, and $3$ (from left to right).}
\label{Fig.DPR.1.Z}
\end{center}
\end{figure}
Despite what the notation suggests, this quantity does not define a norm since it is not absolutely homogeneous. 
In fact, the map $x\mapsto\|x\|_0$ is not convex, or even continuous and as a result, it is often replaced by $x\mapsto\|x\|_1$ in order to make sparse optimization problems computationally tractable \cite{JuditskyNemirovski2020}. Another approach is to replace $\|x\|_0$ by the difference of two of the norms introduced above. Indeed, it is observed in \cite{GaudiosoGorgoneHiriart-Urruty2020,GotohTakedaTono2018} that the sparsity constraint $\|x\|_0\leq{k}$ is equivalent to the equality
$$
\|x\|_{(k)}-\|x\|_{(l)}=0
$$
for any $l$ such that $k<l\leq{d}$. Hence, the norms defined by (\ref{HCP.sec.0.eq.-1}) allow for a computationally effective way to estimate sparsity.

As we mentioned above, the balls for these norms are polytopes. 
The purpose of this article is to study the combinatorics of these polytopes and their duals, by which we mean their $f$-vector, as well as some of their geometric properties such as their volume and the volume of their boundary. 



Throughout the article, we denote by $\gamma_d$ the $d$-dimensional hypercube $[-1,1]^d$ and by $\beta_d$ the cross-polytope whose vertices are the center of the facets of $\gamma_d$, following the notation used by Coxeter \cite{Coxeter1973}. Note that the former is the unit ball for the $\infty$-norm and the latter the unit ball for the $1$-norm. 

We consider the family of polytopes 
\begin{equation}\label{HCP.sec.0.eq.0}
\rho_{d,k}=\mathrm{conv}\!\left(\beta_d\cup\frac{1}{k}\gamma_d\right)
\end{equation}
when $k$ ranges from $1$ to $d$. It is shown in \cite{GaudiosoGorgoneHiriart-Urruty2020} that $\rho_{d,k}$ is the unit ball for the norm defined by (\ref{HCP.sec.0.eq.-1}). Note that these polytopes interpolate between the hypercube and its dual cross-polytope (for another example of polytopes with that property, in a loose sense, see \cite{LeeLeungMargot2004}). In particular, $\rho_{d,1}$ coincides with the hypercube $\gamma_d$ because $\beta_d$ is a subset of that hypercube. Similarly, $\rho_{d,d}$ is equal to the cross-polytope $\beta_d$ as this cross-polytope admits the dilated hypercube $\gamma_d/d$ as a subset. More precisely, the vertices of $\gamma_d/d$ are exactly the center of the facets of $\beta_d$, just as the vertices of $\beta_d$ are the centers of the facets of $\gamma_d$. As we shall see, this observation can be generalized, allowing to determine the whole face lattice of $\rho_{d,k}$. Recall that the norm defined by (\ref{HCP.sec.0.eq.-1}) can also be viewed as the support function of the polytope $\rho_{d,k}^\star$ polar to $\rho_{d,k}$ \cite{Hiriart-UrrutyLemarechal2001}. Since $\gamma_d$ is the polar of $\beta_d$, the polytopes $\rho_{d,k}^\star$ provide another way to continuously deform $\beta_d$ into $\gamma_d$, and give rise to alternate (dual) norms.
\begin{figure}
\begin{center}
\includegraphics{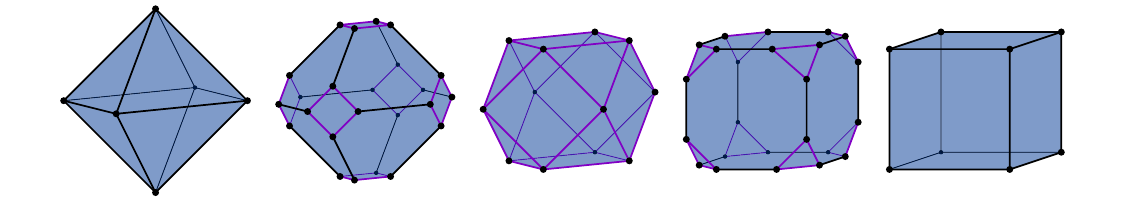}
\caption{The polytope $\rho_{3,k}^\star$, when $k$ is equal to $1$, $3/2$, $2$, $5/2$, and $3$ (from left to right).}
\label{Fig.DPR.1.Z}
\end{center}
\end{figure}
These polytopes have been considered in \cite{BlackDeLoera2021} in the context of monotone path computation within certain polytopal graphs. It follows from (\ref{HCP.sec.0.eq.0}) that
\begin{equation}\label{HCP.sec.0.eq.1}
\rho_{d,k}^\star=(k\beta_d)\cap\gamma_d\mbox{.}
\end{equation}

By duality, the $f$-vector of $\rho_{d,k}^\star$ is obtained by reversing that of $\rho_{d,k}$. However, the volume of $\rho_{d,k}^\star$ requires a separate computation, which we will also provide here. While we are mainly interested in the polytopes $\rho_{d,k}$ and $\rho_{d,k}^\star$ when $k$ is an integer, most of our results hold for any $k$ within the interval $[1,d]$. The combinatorics of $\rho_{d,k}$ and $\rho_{d,k}^\star$ is studied in Section \ref{HCP.sec.1}. The volume of $\rho_{d,k}$ and that of its boundary are computed in Section \ref{HCP.sec.2}. 
The same two volumes, but in the case of $\rho_{d,k}^\star$ rather than $\rho_{d,k}$ are computed in Section \ref{HCP.sec.3}.

\section{The combinatorics of $\rho_{d,k}$ and $\rho_{d,k}^\star$}\label{HCP.sec.1}

As we mentioned earlier, the number of the $(d-i-1)$-dimensional faces of $\rho_{d,k}^\star$ is equal to the number of the $i$-dimensional faces of $\rho_{d,k}$. Therefore, we only need to compute the number of the faces of one of them. In order to do that, we will give a close look at the continuous deformation of $\gamma_d$ into $\beta_d$ via the polytopes given by (\ref{HCP.sec.0.eq.0}). It will be convenient to consider the dilate $k\rho_{d,k}$ instead of $\rho_{d,k}$ itself. Recall, in particular that a polytope has the same combinatorics as any of its dilates by a non-zero coefficient.

First observe that $k\rho_{d,k}$ is obtained by pulling the centers of each facet of the hypercube $\gamma_d$ away from the hypercube along the axes of coordinates, until they are at a distance $k-1$ from the hypercube, and by taking the convex hull of these pulled points together with the vertices of $\gamma_d$. In particular, when $k$ is greater than $1$ but close enough to $1$, $k\rho_{d,k}$ is obtained from the hypercube $\gamma_d$ by glueing pyramids over each of its facets. It immediately follows that, except for its facets, all the proper faces of $\gamma_d$, are still faces of $k\rho_{d,k}$. Moreover, all the facets of $k\rho_{d,k}$ are pyramids over a $(d-2)$-face of $\gamma_d$. More precisely, if $\gamma_{d-2}$ is a $(d-2)$-dimensional face of $\gamma_d$, the two facets of $k\rho_{d,k}$ incident to it are the two pyramids over $\gamma_{d-2}$ whose apices are the points pulled from the center of the facets of $\gamma_d$ incident to $\gamma_{d-2}$. This describes the boundary complex of $k\rho_{d,k}$ whenever $1<k<2$. when $k=2$ the two facets of $k\rho_{d,k}$ incident to $\gamma_{d-2}$ merge into a single facet, a bipyramid over $\gamma_{d-2}$ whose two apices are the points pulled from the center of the facets of $\gamma_d$ incident to $\gamma_{d-2}$. By our description, all the facets of $2\rho_{d,2}$ are built this way. In particular, they are pairwise isometric.

Now let us describe how the boundary complex of $k\rho_{d,k}$ get modified when $2<k<3$. In this case, the two apices of the bipyramid over $\gamma_{d-2}$ are further pulled away from the hypercube, which splits that bipyramid into the convex hulls of the line segment that joins these two apices with each of the facets of $\gamma_{d-2}$. Again, all the facets of $k\rho_{d,k}$ are isometric to such a convex hull. Now recall that the facets of $\gamma_{d-2}$ are $(d-3)$-dimensional faces of $\gamma_d$. In particular these faces are $(d-3)$-dimensional hypercubes. Let $\gamma_{d-3}$ be one of these hypercubes. Observe that $\gamma_{d-3}$ is contained in exactly three facets of $k\rho_{d,k}$ because it is incident to exactly three $(d-2)$-dimensional faces of $\gamma_d$. Upon reaching $k=3$, these three facets merge into a single facet of $3\rho_{d,3}$, obtained as the convex hull of $\gamma_{d-3}$ and the equilateral triangle whose vertices are the points pulled from the centers of the three facets of $\gamma_d$ incident to $\gamma_{d-3}$. As above, all the facets of $3\rho_{d,3}$ are obtained this way.

That process repeats when $k$ belongs to an interval between two consecutive integers. In particular, when $k$ is an integer, each of the $(d-k)$-dimensional faces of the hypercube gives rise to a facet of $k\rho_{d,k}$ and all the facets of $k\rho_{d,k}$ are obtained this way. More precisely, we obtain the following.

\begin{thm}\label{HCP.sec.1.thm.1}
If $k$ is an integer, then the facets of $k\rho_{d,k}$ are exactly the convex hulls of the union of a $(d-k)$-dimensional face $\gamma_{d-k}$ of $\gamma_d$ with the $(k-1)$-dimensional regular simplex whose vertices are the points pulled from the centers of the $k$ facets of $\gamma_d$ incident to $\gamma_{d-k}$.
\end{thm}

It is noteworthy that the regular simplex mentioned in the statement of Theorem~\ref{HCP.sec.1.thm.1} is a face of the cross-polytope $k\beta_d$.

\begin{rem}\label{HCP.sec.1.rem.0}
Let us assume that $k$ is an integer. According to Theorem~\ref{HCP.sec.1.thm.1}, the number of facets of $\rho_{d,k}$ is equal to the number of $(d-k)$-dimensional faces of a $d$-dimensional hypercube. As a consequence,
\begin{equation}\label{HCP.sec.1.rem.0.eq.1}
f_{d-1}(\rho_{d,k})=2^k{d\choose k}\mbox{.}
\end{equation}

By polarity, it further follows from Theorem~\ref{HCP.sec.1.thm.1} that the vertices of $\rho_{d,k}^\star$ are exactly the $2^k{d\choose k}$ points from $\{0,1,-1\}^d$ with $k$ non-zero coordinates. 
\end{rem}

When $k$ is not an integer, the facets of $\rho_{d,k}$ are obtained, in combinatorial terms, by splitting each of the facets of the polytope $\lfloor{k}\rfloor\rho_{d,\lfloor{k}\rfloor}$ into as many facets as a $(d-\lfloor{k}\rfloor)$-dimensional hypercube has.

\begin{thm}\label{HCP.sec.1.thm.2}
If $k$ is not an integer, then the facets of $k\rho_{d,k}$ are exactly the convex hulls of a $(d-\lfloor{k}\rfloor-1)$-dimensional face $\gamma_{d-\lfloor{k}\rfloor-1}$ of $\gamma_d$ with one of the $(\lfloor{k}\rfloor-1)$-dimensional simplices whose vertices are any $\lfloor{k}\rfloor$ of the points pulled from the centers of the $\lfloor{k}\rfloor+1$ facets of $\gamma_d$ incident to $\gamma_{d-\lfloor{k}\rfloor-1}$.
\end{thm}

Based on Theorems \ref{HCP.sec.1.thm.1} and \ref{HCP.sec.1.thm.2}, we now compute the $f$-vector of $\rho_{d,k}$. From there on, we denote by $f_i(P)$ the number of $i$-dimensional faces of a polytope $P$. According to our description, $\rho_{d,k}$ has $2^d+2d$ vertices when $1<k<d$. Recall that $\rho_{d,1}$ is a $d$-dimensional hypercube and $\rho_{d,d}$ a $d$-dimensional cross-polytope, whose number of vertices are $2^d$ and $2d$, respectively. 

As mentioned above, in the case when $k$ is an integer, $f_{d-1}(\rho_{d,k})$ is the number of $(d-k)$-dimensional faces of a $d$-dimensional hypercube. By Theorem \ref{HCP.sec.1.thm.2}, when $k$ is not an integer the number of facets of $\rho_{d,k}$ is the product of the number of $(d-\lfloor{k}\rfloor-1)$-dimensional faces of a $d$-dimensional hypercube with the number of facets of a $\lfloor{k}\rfloor$-dimensional simplex, that is 
$$
f_{d-1}(\rho_{d,k})=2^{\lfloor{k}+1\rfloor}{d\choose \lfloor{k}\rfloor+1}\!\left(\lfloor{k}\rfloor+1\right)\!\mbox{.}
$$

In order to complete the $f$-vector of $\rho_{d,k}$ when $k$ is an integer, let us remark that the faces of $\rho_{d,k}$ are of three types: they can be faces of the hypercube $\gamma_d/k$, faces of the cross-polytope $\beta_d$, or neither. We first compute the number of the $i$-dimensional faces of $\rho_{d,k}$ of the latter type.

\begin{lem}\label{HCP.sec.1.lem.1}
If $k$ is an integer, and $i$ satisfies $1\leq{i}\leq{d-2}$, then the number of the $i$-dimensional faces of $\rho_{d,k}$ that are neither a face of the hypercube $\gamma_d/k$, nor a face of the cross-polytope $\beta_d$ is
$$
{d\choose{k}}\sum_{j=l}^u2^{d-j}\frac{{d-k\choose{j}}{k\choose{i-j}}}{{{d-i}\choose{d-k-j}}}\mbox{,}
$$
where $l=\max\{0,i-k+1\}$ and $u=\min\{i-1,d-k-1\}$.
\end{lem}
\begin{proof}
Assume that $k$ is an integer. According to Theorem \ref{HCP.sec.1.thm.1} any facet of $\rho_{d,k}$ is the convex hull of the union of a $(d-k)$-dimensional face $\gamma_{d-k}$ of the hypercube $\gamma_d/k$ and the $(k-1)$-dimensional regular simplex $\alpha_{k-1}$ whose vertices are the points pulled from the centers of the $k$ facets of $\gamma_d$ incident to $\gamma_{d-k}$. By construction the affine hulls of $\gamma_{d-k}$ and $\alpha_{k-1}$ are orthogonal subspaces of $\mathbb{R}^d$. Therefore, the proper faces of $\mathrm{conv}(\gamma_{d-k}\cup\alpha_{k-1})$ that are not a face of $\gamma_{d-k}$ or a face of $\alpha_{k-1}$ are exactly the convex hulls of the union of a proper face of $\gamma_{d-k}$ and a proper face of $\alpha_{k-1}$. Moreover the dimension this convex hull is greater by one than the sum of the dimension of the faces of $\gamma_{d-k}$ and $\alpha_{k-1}$ it is constructed from. Let us consider a face that arises this way from a $j$-dimensional face $\gamma_j$ of $\gamma_{d-k}$ and a $(i-j-1)$-dimensional face $\alpha_{i-j-1}$ of $\alpha_{k-1}$. Since $\gamma_{d-k}$ is a $(d-k)$-dimensional cube, 
$$
f_j(\gamma_{d-k})=2^{d-k-j}{d-k\choose{j}}\mbox{.}
$$

Since $\alpha_{k-1}$ is a $(k-1)$-dimensional simplex,
$$
f_{i-j-1}(\alpha_{k-1})={k\choose{i-j}}\mbox{.}
$$

Therefore, $\mathrm{conv}(\gamma_{d-k}\cup\alpha_{k-1})$ admits exactly
$$
2^{d-k-j}{d-k\choose{j}}{k\choose{i-j}}
$$
faces of such as $\mathrm{conv}(\gamma_j\cup\alpha_{i-j-1})$. Now observe that if we would multiply this quantity by the number of facets of $\rho_{d,k}$, $\mathrm{conv}(\gamma_j\cup\alpha_{i-j-1})$ would be counted as many times as the number of facets of $\rho_{d,k}$ it is incident to. Let us compute this number. The facets of $\rho_{d,k}$ incident incident to $\mathrm{conv}(\gamma_j\cup\alpha_{i-j-1})$ are obtained by choosing one of the $(d-k)$-dimensional faces $F$ of the hypercube $\gamma_d/k$ incident to $\gamma_j$ and contained in all the facets of $\gamma_d/k$ the vertices of $\alpha_{i-j-1}$ are pulled from, and by then taking the convex hull of its union with the $(k-1)$-dimensional simplex whose vertices are the points pulled from the centers of the $k$ facets of $\gamma_d$ incident to $F$. Hence, there are
$$
{d-i}\choose{d-k-j}
$$
possible choices for $F$, and $\mathrm{conv}(\gamma_j\cup\alpha_{i-j-1})$ is incident to that number of facets of $\rho_{d,k}$. According to these observations, there are
$$
{d\choose{k}}2^{d-j}\frac{{d-k\choose{j}}{k\choose{i-j}}}{{{d-i}\choose{d-k-j}}}
$$
faces of $\rho_{d,k}$ obtained as the convex hull of the union of a $j$-dimensional face of $\gamma_d/k$ with a $(i-j-1)$-dimensional face of $\beta_d$. Such faces of $\gamma_d/k$ and $\beta_d$ exist if and only if $l\leq{j}\leq{u}$ with $l=\max\{0,i-k+1\}$ and $u=\min\{i-1,d-k-1\}$, which completes the proof.
\end{proof}

Now recall that, if $k>1$, then by our description of $\rho_{d,k}$, the hypercube $\gamma_d/k$ shares all of its faces of dimension less than $d-k$ with $\rho_{d,k}$ and no other. Similarly, if $k<d$, then the cross polytope $\beta_d$ shares all of its faces of dimension less than $k-1$ with $\rho_{d,k}$, and no other face.

As a consequence of these observations, we obtain the following.

\begin{lem}\label{HCP.sec.1.lem.2}
If $k>1$, then $\rho_{d,k}$ and $\gamma_d/k$ share
$$
2^{d-i}{d\choose{i}}
$$
faces of dimension $i$ when $0\leq{i}<d-k$ and they do not share any face of dimension $i$ when $d-k\leq{i}<d$. If $k<d$, then $\rho_{d,k}$ and $\beta_d$ share
$$
2^{i+1}{d\choose{i+1}}
$$
faces of dimension $i$ when $0\leq{i}<k-1$ and these polytopes do not have any common face of dimension $i$ when $k-1\leq{i}<d$. 
\end{lem}

The following is an immediate consequence of Lemmas \ref{HCP.sec.1.lem.1} and \ref{HCP.sec.1.lem.2}.

\begin{thm}\label{HCP.sec.1.thm.3}
If $1\leq{i<d-1}$ and $k$ is an integer satisfying $1<k<d$ then\medskip
\begin{itemize}
\item[(i)] $\displaystyle f_i(\rho_{d,k})=f^\star$ when $i\geq\max\{d-k,k-1\}$,\medskip
\item[(ii)] $\displaystyle f_i(\rho_{d,k})=2^{i+1}{d\choose{i+1}}+f^\star$ when $d-k\leq{i}<k-1$,
\item[(iii)] $\displaystyle f_i(\rho_{d,k})=2^{d-i}{d\choose{i}}+f^\star$ when $k-1\leq{i}<d-k$,
\item[(iv)] $\displaystyle f_i(\rho_{d,k})=2^{i+1}{d\choose{i+1}}+2^{d-i}{d\choose{i}}+f^\star$ when $i<\min\{d-k,k-1\}$,
\end{itemize}\medskip
where, in the right-hand side of these equalities,
$$
f^\star={d\choose{k}}\sum_{j=l}^u2^{d-j}\frac{{d-k\choose{j}}{k\choose{i-j}}}{{{d-i}\choose{d-k-j}}}\mbox{,}
$$
with $l=\max\{0,i-k+1\}$ and $u=\min\{i-1,d-k-1\}$.
\end{thm}

\begin{rem}
A conjecture by Kalai \cite{Kalai1989} states that a $d$-dimensional centrally-symmetric polytope always has at least $3^d$ non-empty faces. The polytopes $\rho_{d,k}$ satisfy this conjecture. Indeed, recall that, when $k$ is an integer such that $1<k<d$, $\rho_{d,k}$ has $2d+2^d$ vertices and $2^k{d\choose k}$ facets. Ignoring the $f^\star$ terms in the expression of $f_i(\rho_{d,k})$ provided by Theorem~\ref{HCP.sec.1.thm.3}, one obtains that the number of non-empty faces of $\rho_{d,k}$ is at least
$$
1+2d+2^d+2^k{d\choose{k}}+\sum_{i=2}^{k-1}2^i{d\choose{i}}+\sum_{i=1}^{d-k-1}2^{d-i}{d\choose{i}}\mbox{,}
$$
a sum that can be rearranged into the binomial expansion of $(1+2)^d$.
\end{rem}

\section{The geometry of $\rho_{d,k}$}\label{HCP.sec.2}

Let us recall that $\rho_{d,k}$ is introduced in \cite{GaudiosoGorgoneHiriart-Urruty2020} as an intersection of half-spaces of $\mathbb{R}^d$. We recover this description as a consequence of Theorem \ref{HCP.sec.1.thm.1}.

\begin{cor}\label{HCP.sec.2.cor.1}
If $k$ is an integer, then $\rho_{d,k}$ is the set of the points $x$ in $\mathbb{R}^d$ such that the absolute value of any $k$ coordinates of $x$ sum to at most $1$.
\end{cor}

In the remainder of the section, we compute the volume of $\rho_{d,k}$ and that of its boundary. Let us remark that our description of $k\rho_{d,k}$ naturally provides a polyhedral subdivision of this polytope into convex hulls of unions of hypercubes and simplices. Consider a $(d-l)$-dimensional face $\gamma_{d-l}$ of the hypercube $\gamma_d$ where $l<k$, and the regular $(l-1)$-dimensional simplex $\alpha_{l-1}$ whose vertices are the points pulled from the facets of $\gamma_d$ incident to $\gamma_{d-l}$ in our description of $\rho_{d,k}$. As $l<k$, $\mathrm{conv}(\gamma_{d-l}\cup\alpha_{l-1})$ is a $d$-dimensional polytope. When $l=0$, we will take as a convention that $\gamma_{d-l}$ is the whole hypercube $\gamma_d$ and $\alpha_{l-1}$ is the empty set. The family of these polytopes when $l$ ranges from $0$ to $\lfloor{k}\rfloor$ form a subdivision of $\rho_{d,k}$. This subdivision turns out to be regular. In other words, it can be recovered by projecting to $\mathbb{R}^d$ the lower faces of a $(d+1)$-dimensional polytope \cite{DeLoeraRambauSantos2010}. In this particular case, an example of such a polytope can be obtained by identifying $\mathbb{R}^d$ as the subspace of $\mathbb{R}^{d+1}$ spanned by the first $d$ coordinates, by leaving the vertices of $\gamma_d$ within $\mathbb{R}^d$, by lifting the vertices of $k\beta_d$ in the hyperplane of $\mathbb{R}^{d+1}$ wherein the last coordinate is equal to $1$, and by taking the convex hull of all the resulting points.

According to this discussion, the volume of $\rho_{d,k}$ can be obtained from that of $\mathrm{conv}(\gamma_{d-l}\cup\alpha_{l-1})$. This polytope can be alternatively built by starting from $\gamma_{d-l}$, constructing the pyramid over $\gamma_{d-l}$ whose apex is a vertex of $\alpha_{l-1}$, then taking the pyramid over that pyramid whose apex is another vertex of $\alpha_{l-1}$, and so on until all the vertices of $\alpha_{l-1}$ have been used.

As a consequence, in order to obtain the volume of this polytope, we first compute the distance of a vertex $x$ of $\alpha_{l-1}$ to the affine hull of the union of $\gamma_{d-l}$ and of $i$ vertices of $\alpha_{l-1}$ other than $x$.

\begin{lem}\label{HCP.sec.2.lem.1}
The distance between a vertex $x$ of $\alpha_{l-1}$ and the affine space spanned by $\gamma_{d-l}$ and by $i$ vertices of $\alpha_{l-1}$ other than $x$ is
$$
k\sqrt{\frac{k^2-2(i+1)k+(i+1)l}{k^2-2ik+il}}\mbox{.}
$$
\end{lem}
\begin{proof}
Consider a set $\mathcal{V}$ of $i$ vertices of $\alpha_{l-1}$. We assume without loss of generality that $\gamma_{d-l}$ is the $(d-l)$-dimensional face of the hypercube $\gamma_d$ wherein the first $l$ coordinates are equal to $1$ and that the vertices of $\alpha_{l-1}$ contained in $\mathcal{V}$ are the ones whose positive coordinate is among the first $i$ coordinates. Let us also translate $\gamma_{d-l}$ and $\alpha_{l-1}$ by subtracting the center of $\gamma_{d-l}$, which can be done without loss of generality as well.

In this setting, the affine space spanned by $\gamma_{d-l}\cup\mathcal{V}$ contains the origin. After the translation, the first $i$ coordinates of the points contained in $\mathcal{V}$ are given by the columns of the following matrix, their last $d-l$ coordinates are equal to $0$, and their $l-i$ intermediate coordinates are equal to $-1$:
$$
\left[
\begin{array}{ccccc}
k-1 & -1 & \cdots & -1 & -1\\
-1 & k-1 & \ddots & \vdots & \vdots\\
-1 & -1 & \ddots & -1 & -1\\
\vdots & \vdots & \ddots & k-1 & -1\\
-1 & -1 & \cdots & -1 & k-1\\
\end{array}
\right]\!\mbox{.}
$$

By symmetry, the orthogonal projection of $x$ on the affine space spanned by $\gamma_{d-l}\cup\mathcal{V}$ is a multiple by a coefficient $\lambda$ of the sum of the points contained in $\mathcal{V}$. Now observe that, by symmetry, the equation
$$
\left(x-\lambda\sum_{v\in\mathcal{V}}v\right)\!\cdot{y}=0\mbox{,}
$$
where $y$ is a point from $\mathcal{V}$ does not depend on how $y$ is chosen within $\mathcal{V}$.

Solving that equation for $\lambda$ yields
$$
\lambda=\frac{-2k+l}{k^2-2ik+il}\mbox{.}
$$

Finally, we obtain
$$
\left\|x-\lambda\sum_{v\in\mathcal{V}}v\right\|=k\sqrt{\frac{k^2-2(i+1)k+(i+1)l}{k^2-2ik+il}}\mbox{,}
$$
as desired.
\end{proof}

Observe that the volume of $\gamma_{d-l}$ is $2^{d-l}$. Together with Lemma \ref{HCP.sec.2.lem.1}, the expression of the volume of a pyramid in terms of the volume of its base and the distance of its apex to it, provides the following.

\begin{lem}\label{HCP.sec.2.lem.2}
The volume of $\mathrm{conv}(\gamma_{d-l}\cup\alpha_{l-1})$ is $\displaystyle2^{d-l}k^{l-1}(k-l)\frac{(d-l)!}{d!}$.
\end{lem}

We obtain the volume of $\rho_{d,k}$ from Lemma \ref{HCP.sec.2.lem.2}.

\begin{thm}\label{HCP.sec.2.thm.1}
The volume of $\rho_{d,k}$ is $\displaystyle\frac{2^dk^{\lfloor{k}\rfloor-d}}{\lfloor{k}\rfloor!}$.
\end{thm}
\begin{proof}
By the above remarks on the decomposition of $\rho_{d,k}$ as a polyhedral complex, and since a hypercube has $2^l{d\choose{d-l}}$ faces of dimension $d-l$, it follows from Lemma \ref{HCP.sec.2.lem.2} that the volume of $k\rho_{d,k}$ is
$$
2^d\sum_{l=0}^{\lfloor{k}\rfloor}\frac{k^{l-1}(k-l)}{l!}\mbox{.}
$$

Dividing this quantity by $k^d$, we recover the volume of $\rho_{d,k}$. In addition,
$$
\sum_{l=0}^{\lfloor{k}\rfloor}\frac{k^{l-1}(k-l)}{l!}=\frac{k^{\lfloor{k}\rfloor}}{\lfloor{k}\rfloor!}\mbox{,}
$$
and we obtain the desired result.
\end{proof}

\begin{rem}
It is noteworthy that the volume of $\rho_{d,2}$ does not depend on $d$: by Theorem~\ref{HCP.sec.2.thm.1}, this volume is equal to $2$. Further note that $\rho_{d,2}$ is a non-zonotopal parallelotope (see for instance \cite{Dolbilin2012} and references therein), and that $\rho_{4,2}$ is the $24$-cell, a $4$-dimensional regular, self-dual polytope.
\end{rem}

Let us turn our attention to computing the volume of the boundary of $\rho_{d,k}$. In the remainder of the section, we assume that $k$ is an integer. Since all of the facets of $\rho_{d,k}$ are isometric and we know their number, we only need to compute the volume of one of these facets in order to establish the volume of the boundary of $\rho_{d,k}$. Consider the same $\gamma_{d-l}$ and $\alpha_{l-1}$ as above but, this time, assuming that $k$ and $l$ coincide. In this case, $\mathrm{conv}(\gamma_{d-l}\cup\alpha_{l-1})$ is a facet of $k\rho_{d,k}$. Observe that the polytopes $\mathrm{conv}(\gamma_{d-l}\cup{S})$ where $S$ ranges over the facets of $\alpha_{l-1}$ collectively define a polyhedral subdivision of $\mathrm{conv}(\gamma_{d-l}\cup\alpha_{l-1})$. The volume of these polytopes can be obtained from Lemma \ref{HCP.sec.2.lem.1}. As a consequence, we obtain the volume of $\mathrm{conv}(\gamma_{d-l}\cup\alpha_{l-1})$.

\begin{lem}\label{HCP.sec.2.lem.3}
If $k=l$, then the volume of $\mathrm{conv}(\gamma_{d-l}\cup\alpha_{l-1})$ is 
$$
2^{d-k}k^{k-1/2}\frac{(d-k)!}{d!}\mbox{.}
$$
\end{lem}

Since the number of facets of $\rho_{d,k}$ is $2^k{d\choose{k}}$ and these facets are all isometric, the volume of the boundary of $\rho_{d,k}$ is obtained as an immediate consequence of Lemma \ref{HCP.sec.2.lem.3}. As above, the volume of a facet of $k\rho_{d,k}$ should be divided by $k^{d-1}$ in order to get the volume of a facet of $\rho_{d,k}$.

\begin{thm}\label{HCP.sec.2.thm.2}
If $k$ is an integer, then the volume of the boundary of $\rho_{d,k}$ is
$$
\frac{2^d}{k!}k^{k-d+1/2}\mbox{.}
$$
\end{thm}

\section{The geometry of $\rho_{d,k}^\star$}\label{HCP.sec.3}

By symmetry, the volume of $\rho_{d,k}^\star$ is $2^d$ times the volume of its intersection with the hypercube $[0,1]^d$. That intersection is precisely made up of the points within $[0,1]^d$ whose sum of coordinates is at most $k$. It turns out that an explicit formula is known for the volume of the intersection of $[0,1]^d$ with a half-space bounded by an arbitrary affine hyperplane. 

\begin{thm}[\cite{BarrowSmith1979}]\label{HCP.sec.3.thm.0}
If $a$ is a vector from $\mathbb{R}^d$ whose every coordinate is non-zero, $c$ a real number, and $H^-$ the half space of $\mathbb{R}^d$ made up of the points $x$ satisfying $a\mathord{\cdot}x\leq{c}$, then the volume of $[0,1]^d\cap{H^-}$ is
\begin{equation}\label{HCP.sec.3.thm.0.eq.1}
\frac{\displaystyle\sum(-1)^{\sigma(v)}(c-a\mathord{\cdot}v)^d}{d!\pi(a)}\mbox{,}
\end{equation}
where the sum is over the vertices $v$ of $[0,1]^d$ contained in $H^-$, $\sigma(x)$ stands for the sum of the coordinates of a point $x$ in $\mathbb{R}^d$ and $\pi(x)$ for their product.
\end{thm}

Using this formula, we derive the volume of $\rho_{d,k}^\star\cap[0,1]^d$.
\begin{prop}\label{HCP.sec.3.prop.1}
The volume of $\rho_{d,k}^\star\cap[0,1]^d$ is $\displaystyle\sum_{i=0}^{\lfloor{k}\rfloor}\frac{(-1)^i(k-i)^d}{i!(d-i)!}$.
\end{prop}
\begin{proof}
The desired expression is obtained from Theorem \ref{HCP.sec.3.thm.0}, where $c$ is replaced by $k$ and $a$ by the vector whose coordinates are all equal to $1$. In this case the terms of the sum in the numerator of (\ref{HCP.sec.3.thm.0.eq.1}) only depend on the sum of the coordinates of the associated vertex $v$ of $[0,1]^d$. Rearranging these terms by first summing over the ${d\choose{i}}$ vertices of $[0,1]^d$ whose coordinates sum to $i$ and then letting $i$ range from $0$ to $\lfloor{k}\rfloor$ provides the desired result.
\end{proof}

\begin{rem}
Note that, when $k$ is an integer, $\rho_{d,k}^\star\cap[0,1]^d$ can be naturally decomposed into hypersimplices. More precisely, consider an integer $l$ such that $0\leq{l}<k$. The portion of $[0,1]^d$ made up of the points whose sum of coordinates is between $l$ and $l+1$ is an hypersimplex. It is known \cite{Laplace1886,Stanley1977} that the volume of hypersimplices is obtained by dividing Eulerian numbers \cite{Euler1755} by $d!$. Therefore, when $k$ is an integer, the volume of $\rho_{d,k}^\star\cap[0,1]^d$ can also be expressed in terms of a sum of the Eulerian numbers.
\end{rem}

The volume of $\rho_{d,k}^\star$ is obtained as a consequence of Proposition \ref{HCP.sec.3.prop.1}.

\begin{thm}\label{HCP.sec.3.thm.1}
The volume of $\rho_{d,k}^\star$ is $\displaystyle2^d\sum_{i=0}^{\lfloor{k}\rfloor}\frac{(-1)^i(k-i)^d}{i!(d-i)!}$.
\end{thm}

\begin{rem}
Since $\rho_{d,k}$ and $\rho_{d,k}^\star$ are polar to one another, their Mahler volume is the product of their volumes. Hence, by Theorems \ref{HCP.sec.2.thm.1} and \ref{HCP.sec.3.thm.1}, the Mahler volume of these polytopes is
$$
\frac{4^dk^{\lfloor{k}\rfloor-d}}{\lfloor{k}\rfloor!}\sum_{i=0}^{\lfloor{k}\rfloor}\frac{(-1)^i(k-i)^d}{i!(d-i)!}
$$
when $k$ is an integer. We have computed this quantity up to $d=100$ for all integers $k$ such that $1\leq{k}\leq{d}$ and found that, in these cases, the Mahler volume of $\rho_{d,k}$ is at least $4^d/d!$ as Mahler's conjecture states \cite{IriyehMasataka2020,Mahler1939,NazarovPetrovRyaboginZvavitch2010}.
\end{rem}

Finally, let us compute the volume of the boundary of $\rho_{d,k}^\star$. Unlike $\rho_{d,k}$, the facets of $\rho_{d,k}^\star$ are not pairwise isometric. The facets of $\rho_{d,k}^\star$ contained in a facet of $\gamma_d$ are isometric to $\rho_{d-1,k-1}$, and their volume is given by Theorem~\ref{HCP.sec.3.thm.1}. All the other facets of $\rho_{d,k}^\star$ are isometric to the intersection $\delta_{d-1,k}$ of the hypercube $[0,1]^d$ with the hyperplane made up of the points whose sum of coordinates is equal to $k$. It is noteworthy that, when $k$ is an integer, $\delta_{d-1,k}$ is an hypersimplex and, as we mention above, its volume can be computed from the Eulerian numbers. In fact, a formula for the volume of the intersection of an hypercube with an arbitrary affine hyperplane is established in \cite{FrankRiede2012} based on \cite{Ball1986}.

\begin{thm}[Theorem 2 from \cite{FrankRiede2012}]\label{HCP.sec.3.thm.1.5}
If $a$ is a vector from $\mathbb{R}^d$ whose every coordinate is non-zero, $c$ a real number, and $H$ the half space of $\mathbb{R}^d$ made up of the points $x$ satisfying $a\mathord{\cdot}x=c$, then the volume of $[-1,1]^d\cap{H}$ is
\begin{equation}\label{HCP.sec.3.thm.1.5.eq.1}
\frac{\displaystyle\|a\|_2\sum(a\mathord{\cdot}v+c)^{d-1}s(a\mathord{\cdot}v+c)\pi(v)}{2(d-1)!\pi(a)}\mbox{,}
\end{equation}
where the sum is over the vertices $v$ of $[-1,1]^d$, $s(x)$ stands for the sign of a number $x$, and $\pi(x)$ for the product of the coordinates of a point $x$ in $\mathbb{R}^d$.
\end{thm}

We derive the volume of $\delta_{d-1,k}$ for any $k$ within $[1,d]$ using Theorem \ref{HCP.sec.3.thm.1.5}.

\begin{prop}\label{HCP.sec.3.prop.2}
The volume of $\delta_{d-1,k}$ is $\displaystyle\frac{\sqrt{d}}{(d-1)!}\sum_{i=0}^{\lfloor{k}\rfloor}(-1)^i{d\choose{i}}(k-i)^{d-1}$.
\end{prop}
\begin{proof}
Recall that (\ref{HCP.sec.3.thm.1.5.eq.1}) is a sum over the vertices of $[-1,1]^d$. By a straightforward change of variables, that sum can be transformed into a sum over $\{0,1\}^d$ that provides the volume of the intersection of the hypercube $[0,1]^d$ with a hyperplane. Just as in the proof of Proposition \ref{HCP.sec.3.prop.1}, our special case is such that the terms in that sum only depend on the sum of the coordinates of the point from $\{0,1\}^d$ they correspond to. These terms can therefore be rearranged as we did in the proof of Proposition \ref{HCP.sec.3.prop.1}, by first summing over the points whose coordinates sum to $i$ and then, by letting $i$ range from $0$ to $d$.

The resulting expression for the volume of $\delta_{d,k}$ is
$$
\frac{\sqrt{d}}{2(d-1)!}\!\!\left[\sum_{i=0}^{\lfloor{k}\rfloor}(-1)^i{d\choose{i}}(k-i)^{d-1}-\sum_{i=\lfloor{k}\rfloor+1}^d(-1)^i{d\choose{i}}(k-i)^{d-1}\right]\!\!\mbox{.}
$$

However, it is well-known (see for instance \cite{Ruiz1996}) that
$$
\sum_{i=0}^d(-1)^i{d\choose{i}}(k-i)^{d-1}=0\mbox{.}
$$

As a consequence,
$$
\sum_{i=\lfloor{k}\rfloor+1}^d(-1)^i{d\choose{i}}(k-i)^{d-1}=-\sum_{i=0}^{\lfloor{k}\rfloor}(-1)^i{d\choose{i}}(k-i)^{d-1}\mbox{,}
$$
and the desired result follows.
\end{proof}

Recall that the facets of $\rho_{d,k}^\star$ are either isometric to $\rho_{d-1,k-1}$ or to $\delta_{d-1,k}$. As $\rho_{d,k}^\star$ has $2d$ facets isometric to $\rho_{d-1,k-1}$ and $2^d$ facets isometric to $\delta_{d-1,k}$, we obtain the volume of its boundary from Theorem \ref{HCP.sec.3.thm.1} and Proposition \ref{HCP.sec.3.prop.2}.

\begin{thm}\label{HCP.sec.3.thm.2}
The volume of the boundary of $\rho_{d,k}^\star$ is
$$
2^dd\!\!\left[\sum_{i=0}^{\lfloor{k}\rfloor-1}\frac{(-1)^i(k-1-i)^{d-1}}{i!(d-i-1)!}+\sqrt{d}\sum_{i=0}^{\lfloor{k}\rfloor}\frac{(-1)^i(k-i)^{d-1}}{i!(d-i)!}\right]\!\!\mbox{.}
$$
\end{thm}

\medskip
\noindent{\bf Acknowledgement.} The authors wish to thank Miguel Anjos for initiating this work
by bringing them together, Jun-ya Gotoh for pointing out references \cite{BertsimasPachamanovaSim2004,GotohUryasev2016,WuDingSunToh2014}, and Sebastian Pokutta
and Christoph Spiegel for pointing out references \cite{JuditskyNemirovski2020,Watson1992}. Antoine Deza is partially supported by the Natural Sciences and Engineering Research Council of Canada Discovery Grant Program (RGPIN-2020-06846). Lionel Pournin is partially supported by the ANR project SoS (Structures on Surfaces), grant number ANR-17-CE40-0033.

\bibliography{CubeCrossPolytopeInterpolation}
\bibliographystyle{ijmart}

\end{document}